\documentclass[11pt]{smfart}
\usepackage{color}
\usepackage{amssymb,verbatim}
\usepackage{hyperref}
\usepackage{amsthm,array,amssymb,amscd,amsfonts,amssymb,latexsym, url}
\usepackage{amsmath}
\usepackage[all]{xy}
\usepackage[french]{babel}

\newcounter{spec}
{\end{list}}

\renewcommand{\P}{{\mathbf P}}

\newcommand{\sX}{{\mathcal X}}

\newcommand{\Z}{{\mathbb Z}}
\newcommand{\Q}{{\mathbb Q}}
\newcommand{\C}{{\mathbb C}}
\newcommand{\R}{{\mathbb R}}

\renewcommand{\L}{{\mathbb L}}

\newcommand{\Spec}{{\operatorname{Spec \ }}}

\renewcommand{\lim}{\varprojlim}

\def\ChaL{{Chatzistamatiou et Levine}}

\numberwithin{equation}{section}

\newfont{\gothic}{eufb10}

\renewcommand{\qed}{{\hfill$\square$}}

\newtheorem{theo}{Th\'{e}or\`{e}me}[section]
\newtheorem{prop}[theo]{Proposition}

\newtheorem{lem}[theo]{Lemme}
\newtheorem{cor}[theo]{Corollaire}
\theoremstyle{definition}
\newtheorem{defi}[theo]{D\'efinition}
\theoremstyle{remark}
\newtheorem{rema}[theo]{Remarque}

\newcommand{\bthe}{\begin{theo}}
\newcommand{\ble}{\begin{lem}}
\newcommand{\bpr}{\begin{prop}}
\newcommand{\bco}{\begin{cor}}
\newcommand{\bde}{\begin{defi}}
\newcommand{\ethe}{\end{theo}}
\newcommand{\ele}{\end{lem}}
\newcommand{\epr}{\end{prop}}
\newcommand{\eco}{\end{cor}}
\newcommand{\ede}{\end{defi}}

\newcommand{\F}{{\mathbb F}}

\newcommand{\A}{{\bf A}}

\DeclareFontFamily{U}{wncy}{}
\DeclareFontShape{U}{wncy}{m}{n}{%
<5>wncyr5%
<6>wncyr6%
<7>wncyr7%
<8>wncyr8%
<9>wncyr9%
<10>wncyr10%
<11>wncyr10%
<12>wncyr6%
<14>wncyr7%
<17>wncyr8%
<20>wncyr10%
<25>wncyr10}{}
\DeclareMathAlphabet{\cyr}{U}{wncy}{m}{n}

\begin{document}

  \title[Hypersurfaces cubiques]{Non rationalit\'e stable \\ d'hypersurfaces cubiques \\ sur des corps non alg\'ebriquement clos}

\author{J.-L. Colliot-Th\'el\`ene}
\address{CNRS et  Universit\'e Paris Sud\\Math\'ematiques, B\^atiment 425\\91405 Orsay Cedex\\France}
\email{ }

\date{To appear in the Proceedings of the International Colloquium on $K$-Theory, Mumbai, January 2016.
Paper submitted  June 30th, 2016, accepted September 24th, 2016}
\maketitle

\section*{Introduction}

Soit $F$ un corps.
 Soit $X$ une $F$-vari\'et\'e projective, lisse, g\'eom\'etriquement connexe $X$, telle
que $X(F)\neq \emptyset$.
 On s'int\'eresse aux propri\'et\'es suivantes.
 
 \medskip
 
(i) Rationalit\'e :  La $F$-vari\'et\'e $X$ est $F$-birationnelle \`a un espace  projectif $\P^d_{F}$.

(ii) Rationalit\'e stable :  Il existe un entier $r$ tel  que  la $F$-vari\'et\'e $X \times_{F}\P^r_{F}$
est $F$-birationnellle \`a $\P^{r+d}$.

(iii) R\'etracte rationalit\'e : Il existe un ouvert non vide $U \subset X$,  un entier~$s$,
un ouvert non vide
$V \subset \P^s_{F}$  et un $F$-morphisme $V \to U$ qui admet une section.

(iv) $R$-trivialit\'e : Pour tout corps $L$ contenant $F$, l'ensemble $X_{L}(L)/R$
quotient de $X(L)$ par la R-\'equivalence sur $X_{L}$
a  exactement un \'el\'ement.

(v) $CH_{0}$-trivialit\'e : La $F$-vari\'et\'e $X$ est universellement $CH_{0}$-triviale, c'est-\`a-dire que
 pour tout corps $L$ contenant $F$,
l'application degr\'e $deg_{L } : CH_{0}(X_{L}) \to \Z$ est un isomorphisme
(voir \cite{ACTP, CTP}).

(vi) Trivialit\'e de la cohomologie non ramifi\'ee : 
Pour tout module fini galoisien $M$ sur $F$ d'ordre premier \`a la caract\'eristique de $F$,
pour tout entier $i\geq 0$ et tout corps $L$ contenant $F$, l'application naturelle
$H^{i}(L,M) \to H^{i}_{nr}(L(X)/L,M)$ est un isomorphisme.
 \medskip

Chacune des propri\'et\'es 
 implique la suivante.  Pour le passage de (ii) \`a (iii) sur un corps quelconque, voir  \cite[Cor. 3.3]{merkurjev3}.
Qu'en toute caract\'eristique l'hypoth\`ese (iii) implique (iv) est \'etabli par Kahn et Sujatha  \cite[Thm. 8.5.1 et  Prop. 8.6.2]{kahnsujatha}.
 Pour les d\'efinitions  des diff\'erentes notions et des r\'ef\'erences pour les autres implications, on consultera \cite{kahnsujatha, ACTP, pirutka}.

\medskip
  
Soit $X \subset \P^n_{F}$, avec $n=d+1 \geq 3$ une hypersurface cubique lisse de dimension $d$
poss\'edant un point $F$-rationnel.
Il  est connu qu'une telle hypersurface est $F$-unirationnelle \cite{kollar}.

Pour $F$ alg\'ebriquement clos, $X$ est rationnelle si $d=2$, non rationnelle si $d=3$ 
et ${\rm char}(F)\neq 2$ (Clemens--Griffiths, Mumford, Murre).
Pour $d \geq 4$,  certaines hypersurfaces cubiques  lisses de dimension paire sont rationnelles.
Pour les hypersurfaces cubiques lisses de dimension impaire, on ne sait rien sur la
rationalit\'e, la rationalit\'e stable, ou m\^eme la r\'etracte rationalit\'e.
 Claire Voisin \cite{voisincubic}  a  montr\'e qu'il existe des hypersurfaces cubiques de dimension $d=3$ sur $\C$ qui sont
 universellement $CH_{0}$-triviales. Elle a aussi \'etabli cette derni\`ere propri\'et\'e pour de larges classes d'hypersurfaces cubiques de dimension $d=4$.

 Pour $F$ non alg\'ebriquement clos, que peut-on dire  ?

\medskip

Pour $F$ le corps des r\'eels, B. Segre  \cite{segre} observa qu'une surface cubique lisse $X/\R$
telle que l'espace topologique $X(\R)$ ait deux composantes connexes n'est pas
$\R$-rationnelle. Cette observation se g\'en\'eralise. Ceci est discut\'e au \S \ref{reel}, o\`u l'on montre
que pour tout entier $n \geq 3$ et tout corps $F \subset \R$ il existe des hypersurfaces cubiques lisses $X \subset \P^n_{\F}$
qui ne sont  pas $CH_{0}$-triviales, et en particulier ne sont  pas stablement rationnelles.

\medskip

Que peut-on dire lorsque le corps $F$ n'est pas formellement r\'eel, i.e. lorsque $-1$
est une somme de carr\'es dans $F$ ?

Pour $F$ un corps de caract\'eristique diff\'erente de 2, on note $u(F) \leq \infty$ le plus grand entier $n \leq \infty$
tel qu'il existe une forme quadratique anisotrope de rang $n$ sur $F$.  Rappelons  que l'on a $u(F((t)))=2 u(F)$.

\medskip

Pour $X \subset \P^3_{F}$, on peut utiliser le groupe de Brauer pour donner des exemples
de surfaces cubiques non stablement rationnelles \cite{manin}. On donne facilement de tels exemples
d\'ej\`a sur un corps fini $\F$, sur $\C((x))$, et  sur $\C(x)$. 

Sur $F=\C((x))((y))$,
 D. Madore \cite[Proposition 2.1]{madore} a construit, par un argument de sp\'ecialisation \'elabor\'e, une hypersurface cubique  diagonale $X \subset \P^4_{F}$  
 telle que  l'application degr\'e $CH_{0}(X) \to \Z$ ne soit pas un isomorphisme.
 Ceci implique que $X$ n'est pas  r\'etracte rationnelle.
 
 Sur $F$ un corps    $p$-adique  quelconque et  sur  $F=\F((x))$ (avec $\F$ fini de caract\'eristique diff\'erente de 3),
par sp\'ecialisation \`a une hypersurface cubique singuli\`ere sur un corps fini,
 A. Pirutka et l'auteur \cite[Th\'eor\`eme 1.19]{CTP} ont construit des hypersurfaces cubiques lisses $X \subset \P^4_{F}$ avec un $F$-point qui ne sont pas universellement $CH_{0}$-triviales.
 On en d\'eduit de tels exemples sur tout corps global de caract\'eristique diff\'erente de 3.

 Dans la pr\'epublication r\'ecente  \cite{CL},  par sp\'ecialisation \`a une hypersurface cubique produit d'une quadrique et d'un hyperplan,
 Chatzistamatiou et Levine 
construisent des exemples d'hypersurfaces cubiques lisses  $X \subset \P^n_{F}$ avec un $F$-point,
  non universellement $CH_{0}$-triviales, et donc
non r\'etractes rationnelles,
dans la situation suivante : $k$ un corps de caract\'eristique diff\'erente de 2, 
 $F=k((x))$, $u(k)\geq 2^{\ell}+1$ et  $n=2^{\ell}$.
  Une inspection de leur argument (voir \S \ref{ChatziLevine} ci-dessous)
 montre qu'il suffit en fait de supposer $u(k)\geq 2^{\ell}$.
 On obtient ainsi de tels exemples $X \subset \P^{2^{\ell}}_{F}$
 sur le corps $F=\C((\lambda_{1})) \dots ((\lambda_{\ell+1}))$.

Je propose ici deux autres m\'ethodes pour obtenir des hypersurfaces
cubiques lisses, avec un point rationnel, qui ne sont pas r\'etractes rationnelles.

Au \S \ref{viadphi},
sur $F=k((x))$, avec $k$  un corps de caract\'eristique  diff\'erente de 2, poss\'edant une forme de Pfister anisotrope de dimension $2^{\ell}$,
avec  $\ell \geq 2$, par exemple sur $F=\C((\lambda_{1})) \dots ((\lambda_{\ell+1}))$,
je  construis des exemples d'hypersurfaces cubiques lisses
 $X \subset \P^{n}_{F}$ avec un $F$-point, 
 non r\'etractes rationnelles,
 pour tout entier $n$ avec
$3 \leq n \leq 2^{\ell -1}+1$. Ceci ne couvre pas le cas $n=2^{\ell}$, obtenu par la m\'ethode de
Chatzistamatiou et Levine. L'argument donn\'e  utilise la sp\'ecialisation de la $R$-\'equivalence sur
une fibre sp\'eciale   g\'eom\'etriquement int\`egre mais singuli\`ere en codimension 2 et un succ\'edan\'e
de  cohomologie non ramifi\'ee sur la d\'esingularisation de cette fibre.

Au \S \ref{diagonal}, pour $k$ un corps de caract\'eristique    diff\'erente de 3 poss\'edant un \'el\'ement $a \notin k^{*3}$, par exemple pour $k=\C ((\lambda_{1} ))$, sur
le corps $F=k((\lambda_{2})) \dots ((\lambda_{\ell+1}))$
je  construis des exemples d'hypersurfaces cubiques lisses diagonales
 $X \subset \P^{n}_{F}$ avec un $F$-point, non universellement $CH_{0}$-triviales, donc
 non r\'etractes rationnelles,
 pour tout entier $n$ avec $3 \leq n \leq \ell +3 $.
 L'argument donn\'e utilise la cohomologie non ramifi\'ee, dont on d\'emontre par
 sp\'ecialisations successives qu'elle n'est pas constante.
 
 Au \S 5, on compare les r\'esultats  des \S 2, 3 et 4
 sur les corps de s\'eries formelles it\'er\'ees, d'abord sur les complexes puis
 sur les corps $p$-adiques.

\section{Composantes connexes r\'eelles}\label{reel}

\begin{theo}\label{rappelreel}
Soit $k$ un sous-corps de $\R$.
Soit $X$ une $k$-vari\'et\'e projective, lisse, g\'eom\'etriquement connexe.
Dans chacun des cas suivants :

(a) la $k$-vari\'et\'e $X$ est $k$-rationnelle,

(b) la $k$-vari\'et\'e  $X$ est $k$-r\'etracte rationnelle, 

(c) la $k$-vari\'et\'e   $X$ est universellement $CH_{0}$-triviale, 

\noindent  l'espace topologique $X(\R)$ est
  non vide et connexe.
  \end{theo}
\begin{proof}
Il suffit d'\'etablir le r\'esultat dans le cas $k=\R$.
Le cas (a) est un cas particulier de  (b) qui d'apr\`es \cite[Lemme 1.5]{CTP} est un cas particulier de (c). Si la $\R$-vari\'et\'e $X$ est universellement $CH_{0}$-triviale, en particulier l'application degr\'e $CH_{0}(X) \to \Z$ est un isomorphisme. Ainsi $X$ poss\`ede un z\'ero-cycle de degr\'e 1, et donc un point r\'eel.  Soit $s\geq 1$ le nombre de composantes connexes de $X(\R)$.
D'apr\`es \cite[Prop. 3.2]{CTI} (voir aussi \cite[Thm. 3.1]{fqmva2}),
 pour toute $\R$-vari\'et\'e projective, lisse, g\'eom\'etriquement connexe
avec $X(\R) \neq \emptyset$, on a  $CH_{0}(X)/2= (\Z/2)^{s}$. Ainsi $s=1$.
 \end{proof}
 \begin{rema}
Sous l'hypoth\`ese (a), on peut  \'etablir la connexit\'e de $X(\R)$ par des m\'ethodes
plus classiques.   On montre directement que le nombre de composantes connexes
de $X(\R)$ est un invariant birationnel des $\R$-vari\'et\'es projectives, lisses, g\'eom\'etriquement connexes
(ce qui r\'esulte aussi du r\'esultat sur le groupe de Chow mentionn\'e ci-dessus).
 \end{rema}

\begin{prop}  
Soit $k$ un sous-corps de $\R$.
Pour tout entier $n \geq 2$, il existe une hypersurface cubique lisse $X \subset \P^n_{k}$
telle que l'espace topologique $X(\R)$ ait deux composantes connexes.
Une telle hypersurface n'est pas r\'etracte rationnelle.
\end{prop}
\begin{proof}
 Soit $X \subset \P^{n+1}_{k}$ donn\'ee par l'annulation de 
$$R(x_{1} ,\dots,x_{n},u,v):=  (\sum_{i=1}^n x_{i}^2)v   - u(u-v)(u+ v).$$
Le lieu singulier est donn\'e  par
$u=v=0=\sum_{i=1}^n x_{i}^2=0$, donc ne poss\`ede par de $\R$-point.
Ainsi $X(\R)$ est une vari\'et\'e $C^{\infty}$.  On v\'erifie que cette vari\'et\'e
poss\`ede deux composantes connexes. 
Soit en effet $H \subset \P^{n+1}_{\R}$ l'hyperplan \`a l'infini d\'efini par $v=0$,
et soit $\A^{n+1}_{\R}$ son compl\'ementaire. 
La trace de $X_{\R}$ sur $\A^{n+1}_{\R}$ est donn\'ee par l'\'equation affine
$$ \sum_{i=1}^n x_{i}^2   = u(u-1)(u+ 1)$$
dont le lieu r\'eel est la r\'eunion disjointe d'une partie  born\'ee
satisfaisant    $-1 \leq u \leq 0$ et d'une partie non born\'ee satisfaisant 
 $u \geq 1$. Ceci montre que $X(\R)$ est disconnexe, donc
 a deux composantes connexes,  car c'est le maximum possible pour
 une hypersurface cubique $X$  dans $\P^n_{\R}$ avec $n \geq 2$.

 Soit $S(x_{1} ,\dots,x_{n},u,v)$ une forme cubique sur $\Q$ d\'efinissant une
 hypersurface cubique lisse, par exemple
 $$S(x_{1} ,\dots,x_{n},u,v)=\sum_{i=1}^n x_{i} ^3+ u^3+v^3.$$
Pour $t  \in k \subset \R$ non nul et tr\`es proche de $0$,
  l'hypersurface cubique $X_{t}$ de $\P^{n+1}_{k} $ d\'efinie par
 $R+tS =0$ est lisse. Pour tout  $t\in \R$ assez proche de $0$,
 les vari\'et\'es $C^{\infty}$ donn\'ees par $X_{t}(\R)$ et $X_{0}(\R)=X(\R)$  sont diff\'eomorphes
 (th\'eor\`eme d'Ehresmann), et en particulier ont le m\^eme nombre de composantes connexes.
 La deuxi\`eme partie de l'\'enonc\'e r\'esulte du th\'eor\`eme \ref{rappelreel}.
 \end{proof}

\section{R\'esultats  de    {\ChaL}}\label{ChatziLevine}

Commen\c cons par rappeler un r\'esultat de Totaro
 \cite[Lemme 2.4 
et argument subs\'equent]{totaro}. Pour $X$ une $k$-vari\'et\'e propre,
on note $A_{0}(X)$ le noyau de l'application degr\'e $deg_{k} :CH_{0}(X) \to \Z$.

\begin{lem}
Soit $R$ un anneau de valuation discr\`ete hens\'elien de corps r\'esiduel $k$
 et de corps des fractions $K$.
 Soit $\sX$ un $R$-sch\'ema int\`egre propre et plat. Soit $X= \sX \times_{R}K$.
 Supposons que la fibre sp\'eciale  $Y/k$ est la r\'eunion de deux diviseurs
 $Y_{1} $ et $Y_{2}$ sans composante commune.
 Soit $Z$ le $k$-sch\'ema $ Y_{1} \cap Y_{2}$.
  Supposons $Y_{1}/k$ lisse et $A_{0}(X)=0$.  Alors :
  
  (i)  L'application
 $A_{0}(Z) \to A_{0}(Y_{1})$ est   surjective.
 
 (ii)
 Si de plus  l'indice de $Z$ est \'egal \`a celui de $Y_{1}$,
 par exemple si $Z$ poss\`ede un z\'ero-cycle de degr\'e 1,
alors l'application $CH_{0}(Z) \to CH_{0}(Y_{1})$ est surjective.
  \end{lem}
 
\begin{proof}
 On a une suite exacte
 $$ CH_{0}(Z) \to CH_{0}(Y_{1}) \oplus CH_{0}(Y_{2}) \to CH_{0}(Y) \to 0$$
 o\`u la seconde fl\`eche  envoie un couple $(z_{1},z_{2})$ sur $z_{1}-z_{2}$.

Comme $Y_{1}$ est lisse,  tout z\'ero-cycle de degr\'e z\'ero sur $Y_{1}$ est rationnellement \'equivalent sur $Y_{1}$
 \`a un z\'ero-cycle $z$ de degr\'e z\'ero  \`a support dans le compl\'ementaire de $Z$ dans $Y_{1}$.
 Un tel z\'ero-cycle $z$ se rel\`eve en un z\'ero-cycle de degr\'e z\'ero sur $X$,
dont l'image  par la fl\`eche de sp\'ecialisation
 $ CH_{0}(X) \to CH_{0}(Y)$
 est l'image du couple $(z,0) \in CH_{0}(Y_{1}) \oplus CH_{0}(Y_{2})$.
 L'hypoth\`ese $A_{0}(X)=0$ assure que cette image est nulle.
 La suite exacte ci-dessus \'etablit alors l'existence d'une classe
 $\zeta \in CH_{0}(Z)$ dont l'image est $(z,0) \in CH_{0}(Y_{1}) \oplus CH_{0}(Y_{2})$.
En particulier le  degr\'e de $\zeta$ est z\'ero, ce qui \'etablit l'assertion (i). L'assertion (ii) est alors claire.
\end{proof}
 
Voici une version 
du r\'esultat utilis\'e par {\ChaL}  \cite{CL}.

 \begin{lem}\label{specred}
 Soit $R$ un anneau de valuation discr\`ete   de corps r\'esiduel $k$
 et de corps des fractions $K$.
 Soit $\sX$ un $R$-sch\'ema int\`egre propre et plat. Soit $X= \sX \times_{R}K$.
 Supposons que la fibre sp\'eciale  $Y/k$ est, comme diviseur, la somme de deux diviseurs effectifs 
 $Y_{1}$
  et $Y_{2}$, 
  qui comme $k$-vari\'et\'es sont  lisses et  g\'eom\'etriquement int\`egres.
 Soit $Z$ le $k$-sch\'ema $ Y_{1} \cap Y_{2}$.

  Supposons la $K$-vari\'et\'e $X$ g\'eom\'etriquement int\`egre et universellement
  $CH_{0}$-triviale.   
Si   l'indice de ${Y_{2}}_{   k(Y_{1})   }    $ sur le corps $k(Y_{1})$ est \'egal \`a 1,
  alors l'application $CH_{0}(Z_{k(Y_{1})} ) \to CH_{0}({Y_{1}} _{k(Y_{1})})$ est surjective,
  et l'indice de $ Z_{k(Y_{1})} $ sur $k(Y_{1})$ est \'egal \`a 1.
  \end{lem}

  \begin{proof}
  Il existe un homomorphisme local $R \to S$ d'anneaux de valuation discr\`ete, avec $S$ hens\'elien,
  de corps des fractions $L$
  induisant l'inclusion $k \subset k(Y_{1})$ au niveau des corps r\'esiduels.
  On consid\`ere alors la suite exacte
   $$ CH_{0}(Z_{k(Y_{1}  )}   ) \to CH_{0}(   {Y_{1}}_{k(Y_{1}  )} ) \oplus CH_{0}({Y_{2}} _{k(Y_{1}  )} ) \to CH_{0}(Y_{k(Y_{1}  )} ) \to 0.$$
   
  Soit $z$ un z\'ero-cycle sur ${Y_{1}}_{k(Y_{1}  )}$. Comme $Y_{1}$ est lisse,
  ce z\'ero-cycle est rationnellement \'equivalent \`a un z\'ero-cycle $z_{1} $ \`a support
 \'etranger \`a $Z$.
  Par hypoth\`ese, il
 existe un z\'ero-cycle $w$ sur ${Y_{2}} _{k(Y_{1}  )}$  de degr\'e \'egal \`a celui de $z_{1}$.
Comme $Y_{2}$ est lisse, ce z\'ero-cycle  $w$ est rationnellement \'equivalent sur ${Y_{2}} _{k(Y_{1}  )}$
\`a un z\'ero-cycle $z_{2}$ dont le support est \'etranger \`a $Z$.
Le z\'ero-cycle $z_{1}-z_{2}$ sur $Y_{k(Y_{1}  )}$ est \`a support dans le lieu lisse du
morphisme $\sX \to \Spec(R)$. 
Il se rel\`eve  donc en un z\'ero-cycle de degr\'e z\'ero sur $X_{L}$
(ceci vaut m\^eme sur un corps non parfait).
Comme la $K$-vari\'et\'e $X$ est universellement  $CH_{0}$-triviale,
par sp\'ecialisation \cite[Prop. 2.6]{fulton},
 $z_{1}-z_{2}$ est rationnellement \'equivalent \`a z\'ero sur $Y_{k(Y_{1}  )}$.
 De la suite exacte ci-dessus on tire l'existence d'une classe de z\'ero-cycle sur $Z_{k(Y_{1}  )}$
 d'image la classe de $z$ dans $CH_{0}({Y_{1}}_{k(Y_{1}  )} )$. Ceci \'etablit la premi\`ere
 partie de l'\'enonc\'e.
 Appliquant le r\'esultat au z\'ero-cycle $z$ de  ${Y_{1}}_{k(Y_{1}  )}$ de degr\'e 1 d\'efini par le point
 g\'en\'erique de $Y_{1}$, on obtient la seconde partie de l'\'enonc\'e.
  \end{proof}

 \begin{rema}
 On peut  donner des variantes de l'\'enonc\'e ci-dessus.
 Supposons $Y_{1}/k$  et $Y_{2}/k$   g\'eom\'etriquement int\`egres mais non n\'ecessairement lisses.
 Supposons qu'il existe un z\'ero-cycle de degr\'e 1 \`a support dans le lieu lisse de  $Y_{2}  \setminus Z$.
 Alors l'indice  de   $ Z_{k(Y_{1})} $ sur $k(Y_{1})$ est \'egal \`a 1.
 En effet, le point g\'en\'erique  $\eta$ de $Y_{1}$ d\'efinit un $k(Y_{1})$-point lisse
 de ${Y_{1}}_{k(Y_{1}}  \setminus Z_{k(Y_{1}}$. L'argument ci-dessus \'etablit alors l'existence
 d'un z\'ero-cycle sur $Z_{k(Y_{1})}$ d'image la classe de $\eta$ dans 
 $ CH_{0}({Y_{1}}_{k(Y_{1}  )} )$.
 \end{rema}

 \begin{theo}({\ChaL})
 Soit $k$ un corps de caract\'eristique  diff\'erente de 2. Si sur  le corps $k$ il existe une forme quadratique anisotrope  en $n= 2^{\ell}$ variables, 
alors il existe une hypersurface cubique lisse $X \subset \P^n_{k((t))}$ qui poss\`ede un $k((t))$-point et qui n'est pas
 universellement $CH_{0}$-triviale.
 \end{theo}

 \begin{proof}
 Soit $\ell \geq 1$, soit $k$ un corps, $car(k) \neq 2$, et soit $q(x_{1},\dots,x_{n})$ une forme quadratique anisotrope
 de rang exactement $n=2^{\ell}$. Soit 
  $$q'(x_{0}, \dots, x_{n}) : =q(x_{1},\dots,x_{n}) + x_{0}^2.$$
   Un th\'eor\`eme de Hoffmann \cite{hoffmann}  (dont une nouvelle d\'emonstration fut  donn\'ee par Merkurjev \cite{merkurjev1} au moyen de
  formules du degr\'e \`a la Rost)  
  assure qu'il n'y a pas d'application rationnelle de la quadrique d\'efine par $q'=0$ dans $\P^n_{k}$ vers la
 quadrique d\'efinie par $q=0$ dans $\P^{n-1}_{k}$. Il n'y a pas besoin ici de supposer que la forme quadratique $q'$ est
 anisotrope.
 Soit $Y_{1} \subset \P^n_{k}$ la quadrique lisse d\'efinie par $q'=0$. Soit $Y_{2} \subset \P^n_{k}$
 l'hyperplan d\'efini par $x_{0}=0$. 
 Alors $Z=Y_{1} \cap Y_{2} \subset Y_{2}$ est la quadrique  d\'efinie  par $q=0$ dans $Y_{2} \simeq \P^{n-1}_{k}$.  
 Il existe une forme cubique lisse en $n+1$ variables sur $k((t))$ qui se sp\'ecialise
 en $t=0$ sur la forme cubique $q'(x_{0}, \dots, x_{n}).x_{0}$, et qui poss\`ede 
un z\'ero non trivial sur $k((t))$, car aucun $k$-point de $Y_{2}$ n'est situ\'e sur $Z$.
  Soit $X \subset \P^n_{k((t))}$
 l'hypersurface cubique lisse qu'elle d\'efinit.
 Si l'indice de $Z$ sur $k(Y_{1})$ \'etait \'egal \`a 1, alors par un  th\'eor\`eme bien connu de Springer,
  la quadrique $Z$ aurait un $k(Y_{1})$-point, donc il y aurait une application rationnelle de $Y_{1}$
  vers $Z$, contredisant le th\'eor\`eme d'Hoffmann. 
   Le lemme \ref{specred} permet alors de conclure que  l'hypersurface cubique $X \subset \P^n_{k((t))}$ n'est pas universellement
  $CH_{0}$-triviale.
 \end{proof}

 \bigskip

\section{Formes quadratiques multiplicatives et R-\'equivalence}\label{viadphi}

\subsection{Certaines hypersurfaces cubiques}\label{chatelet}

 Soit $k$ un corps de caract\'eristique  diff\'erente de 2. Soit  $q(x_{1},\dots,x_{n})$, $n \geq 2$ une forme quadratique non d\'eg\'en\'er\'ee sur $k$. Soit $\rho \in k^*$,
 $\rho \neq 0, \rho \neq 1$.

Soit  $X =X(q,\rho) \subset \P^{n+1}_{k}$ l'hypersurface cubique donn\'ee par 
l'\'equation homog\`ene\footnote{Pour $n=2$, on obtient une surface de Ch\^atelet.}
$$q(x_{1}, \dots, x_{n}) v = u(u-v)(u-\rho v).$$
 Le lieu de non lissit\'e  de $X$ est donn\'e par 
 $u=v=0=q=0$. 
 
Suivant 
 \cite[\S 5]{fqmva1}, soit 
$Y=Y(q,\rho)$ 
le fibr\'e en quadriques sur $\P^1_{k}$  obtenu par recollement de
la vari\'et\'e $Y_{1}$ d\'efinie par
$$q(X_{1}, \dots, X_{n}) =U(U-1)(U-\rho)T^2$$
dans  $\P^n_{k} \times_{k} \A^1_{k}$
muni des coordonn\'ees $(X_{1},\dots,X_{n},T; U)$
et de la vari\'et\'e $Y_{2}$ d\'efinie par
$$q(X'_{1}, \dots, X'_{n}) =V(1-V)(1-\rho V)T^2$$
dans  $\P^n_{k} \times_{k} \A^1_{k}$
muni des coordonn\'ees 
 $(X'_{1},\dots,X'_{n},T; V)$, le recollement se faisant via 
 $V=1/U$ et  $(X'_{1},\dots,X'_{n},T)=(U^{-2}X_{1},\dots,U^{-2}X_{n},T).$ 
 
Les $k$-vari\'et\'es  $X$ et $Y$ sont $k$-birationnellement isomorphes,
 elles contiennent toutes deux la $k$-vari\'et\'e affine lisse $W$ d'\'equation
$$q(x_{1}, \dots, x_{n})  = u(u-1)(u-\rho).$$

Plus pr\'ecis\'ement, on d\'efinit un morphisme birationnel
de $X \setminus \{v=0\}$ vers $Y_{1}$
par
$$(X_{1}, \dots, X_{n}, T; U) = (x_{1}, \dots, x_{n},v ; u/v).$$
On d\'efinit un morphisme birationnel de $X \setminus \{u=0\}$ vers $Y_{2}$
par
$$(X'_{1}, \dots, X'_{n}, T; V) = (x_{1}v,\dots,x_{n}v, u^2; v/u).$$

Ceci d\'efinit un morphisme de $X \setminus \{u=v=0\}$ vers $Y$.

Montrons que ce morphisme s'\'etend en un morphisme de 
$$X_{lisse}=X \setminus \{u=v=q=0\}$$ vers $Y$.
Il suffit d'\'etablir cela au voisinage des points de $u=v=0$ qui satisfont $q\neq 0$.
Sans perte de g\'en\'eralit\'e, on peut  alors se placer sur l'ouvert affine $x_{1} \neq 0$ de $X$.
On s'int\'eresse donc \`a la vari\'et\'e affine d\'efinie par l'\'equation
$$q(1, x_{2}, \dots, x_{n}) = u(u-v)(u-\rho v).$$
Cette \'equation se r\'e\'ecrit
$$[q(1,x_{2}, \dots,x_{n}) +(1+\rho)u^2-\rho uv].v = u^3.$$
Au voisinage d'un point de $u=v=0$ avec $q\neq 0$
la fonction $$f(x_{2}, \dots,x_{n},u,v) :=   q(1,x_{2}, \dots,x_{n}) +(1+\rho)u^2-\rho uv$$ est
inversible.
L'application  rationnelle de $X$ vers $Y_{2}$
est donn\'ee sur l'ouvert  affine d\'efini par $x_{1}\neq 0$ par
$$(X'_{1}, \dots, X'_{n}, T; U) = (v, x_{2}v, \dots,x_{n}v, u^2; v/u).$$
Comme on a $fv=u^3$, cette application est donc aussi donn\'ee par 
$$(X'_{1}, \dots, X'_{n}, T; U) = (u, x_{2}u, \dots,x_{n}u, f; u^2.f^{-1})$$
qui est un morphisme l\`a o\`u $f$ est inversible, donc dans le voisinage de tout point avec
$u=v=0$ et $q\neq 0$.

On notera que l'image de $u=v=0$, $q\neq 0$ est le point $(0,\dots,0,1; 0)$ de $Y_{2} \subset Y$.

Soit d\'esormais $\theta: X_{lisse} \to Y$ le $k$-morphisme birationnel construit ci-dessus.

Soient $A \in W(k) \subset X(k)$ le  $k$-point lisse d\'efini par  $(x_{1}, \dots,x_{n},u,v)= (0, \dots,0, 1,1)$ et
$B \in W(k) \subset X(k)$ le  $k$-point lisse d\'efini par   $(x_{1}, \dots,x_{n},u,v)= (0, \dots,0, \rho,1)$.  Sur $Y(k)$, les images de ces points  par $\theta$ sont donn\'es respectivement par
 $A_{1}=\theta(A)  \in Y_{1}(k)$ de coordonn\'ees $(X_{1},\dots,X_{n},T; u)=(0,\dots,0,1;1)$
et  $B_{1}=\theta(B)  \in Y_{1}(k)$  de coordonn\'ees $(X_{1},\dots,X_{n},T; u)=(0,\dots,0,1;\rho)$.

\begin{lem}\label{Requivsing}
Supposons la forme quadratique $q$ anisotrope sur $k$. 
Si les points $A,B \in W(k)$ sont 
$R$-\'equivalents sur $X$, alors $A_{1}=\theta(A)$ et $B_{1}=\theta(B)$ sont $R$-\'equivalents sur $Y$.
\end{lem}
\begin{proof}
 L'hypoth\`ese sur $q$ garantit que le lieu singulier de $X$, qui est  donn\'e par
 $u=v=q=0$, ne contient aucun $k$-point. On a  donc $X_{lisse}(k)=X(k)$.
 Soient $P$ et $Q$ deux points rationnels de $X(k)$ R-\'equivalents sur $X$.
 Par d\'efinition, il existe une famille  finie de $k$-morphismes  $\sigma_{i} :\P^1_{k} \to X$, $i=1, \dots,m-1$
 avec $\sigma_{i}(0)=P_{i}, \sigma_{i}(\infty)= P_{i+1}$, et $P_{0}=P$ et $P_{m}=Q$.
 Comme $X_{lisse}(k)=X(k)$, les points $P_{i}$ sont tous dans $X_{lisse}$,
 les images des morphismes $\sigma_{i}$ rencontrent $X_{lisse}$.
 La composition avec $\theta$ d\'efinit des applications rationnelles $\rho_{i}$
 de $\P^1_{k}$ vers $Y$, qui puisque $Y$ est projectif, sont en fait des morphismes.
 On a $\rho_{i}(0)= \sigma(P_{i})$ et $\rho_{i}(\infty)= \sigma(P_{i+1})$.
 Ainsi $\theta(P)$ et $\theta(Q)$ sont $R$-\'equivalents sur $Y$.
   \end{proof}

\subsection{Rappel}

Soit $k$ un corps, ${\rm car}.(k) \neq 2$, et soit $\phi$
une $n$-forme de Pfister anisotrope sur $k$.  On sait  (Pfister, voir \cite{lam}) que l'ensemble $N_{\phi}(k) \subset k^*$
des valeurs non nulles de $\phi$ est un sous-groupe de $k^*$.
Soit $Y$ une $k$-vari\'et\'e projective, lisse, g\'eom\'etriquement connexe.
Soit $k(Y)$ son corps des fonctions.  

Dans  \cite{fqmva1,fqmva2}, on a consid\'er\'e le sous-groupe de $k(Y)^*$ form\'e des fonctions qui partout localement sur $Y$
peuvent s'\'ecrire comme le produit d'une unit\'e locale et d'un \'el\'ement de $N_{\phi}(k(Y))$.
On a d\'efini et \'etudi\'e le quotient  $D^{\phi}(Y)$ de ce groupe par le sous-groupe $N_{\phi}(k(Y))$.
 
Il y a en particulier un accouplement
$$ Y(k) \times D^{\phi}(Y) \to D^{\phi}(k)= k^*/N_{\Phi}(k)$$
qui passe au quotient par la $R$-\'equivalence sur $Y(k)$
(\cite[Cor. 4.1.5]{fqmva1},
cons\'equence de la fonctorialit\'e et de l'invariance homotopique des groupes  $D^{\phi}(Y)$).

L'invariance birationnelle des groupes $D^{\phi}(Y)$  fut \'etablie par M. Rost \cite{rost}.

\subsection{Sur un corps de $u$-invariant fini}

On renvoie \`a \cite{lam} pour les propri\'et\'es de base des formes quadratiques
sur les corps.

\begin{theo}
Soit $k$ un corps de caract\'eristique diff\'erente de 2 tel que $I^{n+1}k \neq 0$, i.e.
tel qu'il existe une $(n+1)$-forme de Pfister anisotrope. Soit $F=k((t))$.
Pour tout entier  $m$ avec $2 \leq m \leq 2^n$, il existe une hypersurface cubique  lisse
  $X \subset \P^{m+1}_{F}$ avec un $F$-point
  qui n'est pas r\'etracte rationnelle, en particulier n'est pas stablement rationnelle.
 \end{theo}

\begin{proof}
La condition sur $k$ \'equivaut 
\`a l'existence d'une
  $n$-forme de Pfister $\phi$, de dimension $2^n$ et d'un \'el\'ement $\rho \in k^*$
non repr\'esent\'e par la  forme $\phi$ sur $k$.  Elle implique aussi que le $u$-invariant de $k$ est au moins \'egal \`a $2^{n+1}$.

Soit $q(x_{1}, \dots,x_{m})$, avec $m\geq 2$, une forme quadratique
non d\'eg\'en\'er\'ee sur $k$, de dimension $m \leq 2^n$, qui est une sous-forme de la forme $\phi$ sur $k$. En particulier $q$ est anisotrope.

Soit $Y=Y(q,\rho)$ comme au \S \ref{chatelet}.
On v\'erifie facilement  (cf. \cite[\S 5]{fqmva1}) que la fonction rationnelle $u \in k(W)^*=k(Y)^*$ d\'efinit un \'el\'ement de 
$D^{\phi}(Y)$, et que l'\'evaluation de cet \'el\'ement sur $A_{1}$, resp. $B_{1}$,  est $1 \in k^*/D^{\phi}(k)$, resp. 
$\rho \in k^*/D^{\phi}(k)$. Par hypoth\`ese,  on a $\rho \notin D^{\phi}(k)$.
 Ainsi  les $k$-points $A_{1}$ et $B_{1}$ ne sont pas
$R$-\'equivalents sur $Y$.
(On voit aussi que  l'application naturelle $D^{\phi}(k) \to D^{\phi}(Y)$ n'est pas un isomorphisme.)

Le lemme \ref{Requivsing} garantit alors que les points $A$ et $B$  de $X(k)$ ne sont pas $R$-\'equivalents sur 
l'hypersurface cubique singuli\`ere $X$.

Soit $\Psi(x_{1}, \dots, x_{m}, u,v)=0$ une forme cubique sur le corps $k$ d\'efinissant une 
hypersurface lisse dans $\P^{m+1}_{k}$. 
Soit 
Soit $\mathcal{X} \subset \P^{m+1}_{k[[t]]}$ 
d\'efini par 
$$q(x_{1},\dots,x_{m})v- u(u-v)(u-\rho v)+ t \Psi(x_{1}, \dots, x_{m}, u,v) =0.$$
La fibre sp\'eciale $\mathcal{X}_{0}$ est l'hypersurface cubique  $X$ sur le corps $k$ discut\'ee au~\S \ref{chatelet}.
Elle contient les points $k$-rationnels lisses $A$ et $B$.
Par le lemme de Hensel, le point $A$ se rel\`eve en des $k((t))$-points $\tilde{A}$  de la fibre g\'en\'erique
$\mathcal{X}_{t}$ sur le  corps $F=k((t))$. En outre, l'ensemble de tels   $\tilde{A}$ 
est Zariski dense sur  la $F$-vari\'et\'e $\mathcal{X}_{t}$. On a le m\^{e}me \'enonc\'e  pour 
les relev\'es $\tilde{B}$ de $B$.

Supposons  la $F$-vari\'et\'e  $\mathcal{X}_{t}$    r\'etracte rationnelle. Alors il existe des ouverts non vides $U \subset \mathcal{X}_{t}$ et $ V \subset \P^n_{F}$ et des $F$-morphismes
$s:  U \to V$ et $p : V \to U$ dont le compos\'e $p\circ s$ est l'identit\'e de $U$.
Soient alors $\tilde{A}$ et $\tilde{B}$ des relev\'es de $A$ et $B$ dans $U$. 
Il existe une droite $\P^1_{F} \subset \P^n_{F}$ contenant $s(\tilde{A})$ et 
$s(\tilde{B})$, et qui donc rencontre $V$.
En composant avec $p : V \to U \subset \mathcal{X}_{t}$, puisque $\mathcal{X}_{t}$ est propre et que 
$\P^1_{F} $ est une courbe r\'eguli\`ere,
 on obtient un $F$-morphisme
$\P^1_{F} \to \mathcal{X}_{t}$ tel que $\tilde{A}$ et $\tilde{B}$  soient dans l'image
$\P^1(F)$. Ainsi  $\tilde{A}$ et $\tilde{B}$ sont R-\'equivalents sur la $F$-vari\'et\'e 
$\mathcal{X}_{t}$. On sait \cite{mad1} 
que la R-\'equivalence se sp\'ecialise (ceci vaut pour tout $k[[t]]$-sch\'ema propre): 
l'application $$\mathcal{X}_{t}(F) = {\mathcal X}(k[[t]]) \to X_{0}(k)=X(k)$$ induit une application
$\mathcal{X}_{t}(F)/R \to X(k)/R$. Mais $A$ et $B$ ne sont pas R-\'equivalents sur $X$.
On conclut  (voir l'introduction de l'article) que la $F$-hypersurface cubique lisse $\mathcal{X}_{t}$ n'est pas r\'etracte rationnelle.
 \end{proof}
 
 \begin{rema}
 On peut envisager diverses variantes de la d\'emonstration.

 (1) Tout en utilisant l'argument de  sp\'ecialisation de la R-\'equivalence comme ci-desssus,
 on peut remplacer les groupes $D^{\phi}(Y)$ par la cohomologie non ramifi\'ee.
  Soit $\phi = <1,-a_{1}> \otimes  \dots \otimes <1,-a_{n}>$. Le cup-produit $$(a_{1}) \cup \dots (a_{n}) \cup (u) \in H^{n+1}(k(Y),\Z/2),$$
o\`u pour $b \in k(Y)^*$ on note $(b) \in k(Y)^*/k(Y)^{*2}=H^1(k(Y),\Z/2)$   la classe de $b$, est une classe non ramifi\'ee qui est non constante car elle prend des valeurs distinctes dans $H^{n+1}(k,\Z/2)$.
Ceci utilise  le th\'eor\`eme  (\cite{OVV}) 
qu'une $(n+1)$-forme de Pfister non triviale sur le corps $k$ a une image non nulle dans  $H^{n+1}(k,\Z/2)$.

\medskip

(2)  On peut envisager une m\'ethode qui ignore la R-\'equivalence, et utilise seulement l'\'equivalence rationnelle sur les z\'ero-cycles.

Supposons que l'on  ait :

(*) Il existe un $k$-morphisme de d\'esingularisation $f : X' \to X$  qui est   universellement $CH_{0}$-trivial.

Si la $F$-vari\'et\'e  $\mathcal{X}_{t}$ est r\'etracte rationnelle, alors
elle est universellement $CH_{0}$-triviale \cite[Lemme 1.3]{CTP}.
D'apr\`es \cite[Th\'eor\`eme 1.12]{CTP}, sous l'hypoth\`ese (*), ceci implique   que la $k$-vari\'et\'e $X'$ est universellement $CH_{0}$-triviale. 
Ceci implique alors \cite{merkurjev2} que les groupes de cohomologie non ramifi\'ee \`a coefficients $\Z/2$ de $X'$, qui sont des invariants $k$-birationnels des $k$-vari\'et\'es projectives et lisses,
sont r\'eduits \`a la cohomologie du corps de base. 
Mais on sait (via le mod\`ele $Y$) que ce n'est pas le cas,
car on a une classe de cohomologie non ramifi\'ee qui prend des
valeurs distinctes en deux points -- ce dernier point utilisant
comme ci-dessus \cite{OVV}. 
\end{rema}

\section{Hypersurfaces cubiques diagonales et cohomologie non ramifi\'ee}\label{diagonal}

\begin{theo}\label{theodiagonal}
Soit $k$ un corps de caract\'eristique diff\'erente de 3, poss\'edant un \'el\'ement $a$
qui n'est pas un cube. Soient $0 \leq n \leq m$ des entiers.
Soit $F$ un corps   avec
$$k(\lambda_{1}, \dots, \lambda_{m}) \subset F \subset F_{m}:=k((\lambda_{1})) \dots ((\lambda_{m})).$$
L'hypersurface cubique      $X:= X_{n,F}$ de $\P^{n+3}_{F}$ d\'efinie par l'\'equation
$$ x^3+y^3+z^3+aw^3+ \sum_{i=1}^n \lambda_{i} t_{i}^3=0$$
 poss\`ede un point rationnel et 
n'est pas
 universellement $CH_{0}$-triviale,  en particulier elle  n'est pas r\'etracte rationnelle.
\end{theo}

\begin{proof} 
Pour \'etablir le r\'esultat, on peut supposer que $k$ contient une racine cubique primitive de l'unit\'e, soit $j$,
et que $F=F_{m}$.
Le lemme  \ref{k((t))}  ci-dessous permet de supposer $n=m$.
On fixe un isomorphisme $\Z/3 = \mu_{3}$ et on consid\`ere la cohomologie \'etale 
\`a coefficients $\Z/3$.  On ignore les torsions \`a la Tate dans les notations.
Etant donn\'es un corps $L$ contenant $k$ et des \'el\'ements $b_{i}, i=1, \dots, s,$ de $L^*$,
on note 
$(b_{1}, \dots, b_{s}) \in H^s(L,\Z/3)$ le cup-produit, en cohomologie galoisienne,
des classes $(b_{i}) \in L^*/L^{*3}=H^1(L,\Z/3)$.

 On va d\'emontrer par
r\'ecurrence sur $n\neq 0$  l'assertion suivante, qui implique la proposition.

($A_{n}$) 
 Soient $k$, $a$, $F_{n}$ et $X_{n}/F_{n}$ comme ci-dessus.   
Le cup-produit
$$\alpha_{n}: = ((x+jy)/(x+y), a, \lambda_{1},\dots,\lambda_{n}) \in H^{n+2}(F_{n}(X_{n}),\Z/3)$$
d\'efinit une classe de cohomologie non ramifi\'ee  (par rapport au corps de base $F_{n}$) qui ne provient pas
d'une classe dans $H^{n+2}(F_{n},\Z/3)$.

Le cas $n=0$ est connu (\cite[Chap. VI, \S 5]{manin} \cite[\S 2.5.1]{CTS}).
Supposons l'assertion d\'emontr\'ee pour $n$.

La classe $\alpha_{n+1}$ sur la $F_{n+1}$-hypersurface $X_{n+1} \subset \P^{n+4}_{F_{n+1}}$ a ses r\'esidus 
triviaux en dehors des diviseurs d\'efinis par $x+y=0$ et $x+jy=0$.
Soit $\Delta \subset X_{n+1}$ le diviseur $x+y=0$. Ce diviseur est d\'efini par les
\'equations 
$$x+y=0, z^3+aw^3+ \sum_{i=1}^{n+1} \lambda_{i} t_{i}^3=0.$$
Le r\'esidu  de $\alpha_{n+1}$ au point g\'en\'erique de $\Delta$ est
 $$\partial_{\Delta}(\alpha_{n+1}) = 
  \pm (a, \lambda_{1},\dots,\lambda_{n+1}) \in H^{n+2}(F_{n+1}(\Delta),\Z/3).$$
 Mais dans le corps des fonctions de $\Delta$, on a 
$$1+a(w/z)^3+ \sum_{i=1}^{n+1} \lambda_{i} (t_{i}/z)^3=0$$
et 
cette \'egalit\'e implique (cf. \cite[Lemma 1.3]{milnor})  :
$$(a,\lambda_{1}, \dots, \lambda_{n+1}) = 0 \in  H^{n+2}(F_{n+1}(\Delta),\Z/3).$$
Le m\^{e}me argument s'applique pour le diviseur
d\'efini par  $x+jy=0$. Ainsi $\alpha_{n+1}$ est une classe de cohomologie non
ramifi\'ee sur la $F_{n+1}$-hypersurface $X_{n+1}$.

Soit   $\mathcal{X}_{n+1}$ le $F_{n}[[\lambda_{n+1}]]$-sch\'ema d\'efini par
$$ x^3+y^3+z^3+aw^3+ \sum_{i=1}^{n+1} \lambda_{i} t_{i}^3=0.$$
Le diviseur $Z$ d\'efini par $\lambda_{n+1}=0$ sur $\mathcal{X}$
est le c\^{o}ne d'\'equation
$$ x^3+y^3+z^3+aw^3+ \sum_{i=1}^n \lambda_{i} t_{i}^3=0$$
dans $\P^{n+4}_{F_{n}}$, c\^one qui est birationnel au produit de $\P^1_{F_{n}}$
et de l'hypersurface cubique lisse  $X_{n} \subset \P_{F_{n}}^{n+3}$ d\'efinie par
$$ x^3+y^3+z^3+aw^3+ \sum_{i=1}^n \lambda_{i} t_{i}^3=0.$$
 Le corps des fonctions rationnelles de $\mathcal{X}_{n+1}$ est
$F_{n+1}(X_{n+1})$.

On a  $$\partial_{Z}(\alpha_{n+1} ) =   \pm  ((x+jy)/(x+y), a, \lambda_{1},\dots,\lambda_{n})  \in H^{n+2}(F_{n}(Z), \Z/3).$$

Par l'hypoth\`ese de r\'ecurrence 
$$ ((x+jy)/(x+y), a, \lambda_{1},\dots,\lambda_{n}) \in H^{n+2}(F_{n}(X_{n}), \Z/3)$$
n'est pas dans l'image de $ H^{n+2}(F_{n},\Z/3)$. Ceci implique que
$$ ((x+jy)/(x+y), a, \lambda_{1},\dots,\lambda_{n}) \in H^{n+2}(F_{n}(Z)), \Z/3)$$
n'est pas dans l'image de $H^{n+2}(F_{n},\Z/3)$.
Du  diagramme commutatif 
$$\begin{array}{ccccccccc}
 \partial_{Z} : & H^{n+3}(F_{n+1}(X),\Z/3)  & \to & H^{n+2}(F_{n}(Z), \Z/3)\\
 & \uparrow & &  \uparrow \\
 \partial_{\lambda_{n+1}=0} : & H^{n+3}(F_{n+1},\Z/3) & \to  &H^{n+2}(F_{n},\Z/3)
 \end{array}$$
on conclut que 
$$\alpha_{n+1}: = ((x+jy)/(x+y), a, \lambda_{1},\dots,\lambda_{n+1}) \in H^{n+3}(F_{n+1}(X),\Z/3)$$
n'est pas dans l'image de $H^{n+3}(F_{n+1},\Z/3)$.

Ceci \'etablit $(A_{n})$ pour tout entier $n$ et
 implique (cf. \cite{merkurjev2}) que la  $F_{n}$-vari\'et\'e $X_{n}$ n'est pas universellement $CH_{0}$-triviale  et   n'est pas r\'etracte rationnelle.
\end{proof}

\begin{lem}\label{k((t))}
Soit $F$ un corps. Si  une $F$-vari\'et\'e $X$
projective, lisse, g\'eom\'etriquement connexe n'est pas universellement $CH_{0}$-triviale, alors la $F((t))$-vari\'et\'e 
$X\times_{F}F((t))$ n'est pas universellement $CH_{0}$-triviale, et donc n'est pas r\'etracte rationnelle.
\end{lem}
\begin{proof}
Sur tout corps $L$ contenant $F$, on dispose de l'application de sp\'ecialisation
$CH_{0}(X_{L((t))}) \to CH_{0}(X_{L})$, et cette application est surjective et respecte
le degr\'e.
\end{proof}

\begin{rema}
Il serait int\'eressant de comprendre la g\'en\'eralit\'e de la construction faite dans le
th\'eor\`eme \ref{theodiagonal}. On utilise une classe de cohomologie non ramifi\'ee non constante sur un mod\`ele
birationnel de la fibre sp\'eciale d'une $k[[t]]$-sch\'ema propre \`a fibres int\`egres, et on en tire une classe de cohomologie non ramifi\'ee non constante
de degr\'e un de plus sur la fibre g\'en\'erique sur $k((t))$, essentiellement par cup-produit avec la classe d'une uniformisante
de l'anneau $k[[t]]$.
\end{rema}

On laisse au lecteur le soin d'\'etablir l'analogue suivant du th\'eor\`eme \ref{theodiagonal}.

\begin{theo}\label{theodiagonalsurpadique}
Soient $p \neq 3$ un nombre premier et $k$ un corps $p$-adique dont le corps r\'esiduel
contient les racines cubiques primitives de 1. Soit $a \in k^*$ une unit\'e qui n'est pas un cube.
Soit $\pi$ une uniformisante de $k$. Soient $0 \leq n \leq m$ des entiers.
Soit $F$ un corps   avec
$$\Q(a) (\lambda_{1}, \dots, \lambda_{m}) \subset F \subset k((\lambda_{1})) \dots ((\lambda_{m})).$$
L'hypersurface cubique      $X_{n}$ de $\P^{n+4}_{F}$ d\'efinie par l'\'equation
$$ x^3+y^3+z^3+aw^3+ \pi t^3+  \sum_{i=1}^n \lambda_{i} t_{i}^3=0,$$
qui poss\`ede un point rationnel,
n'est pas
 universellement $CH_{0}$-triviale et  donc n'est pas r\'etracte rationnelle.
\end{theo}

\bigskip

{\it Exemples}

 En appliquant le th\'eor\`eme \ref{theodiagonal},
on trouve  $X_{n} \subset \P^{n+3}_{F}$ non r\'etracte rationnelle
avec 
$$k(\lambda_{1}, \dots, \lambda_{n}) \subset F \subset k((\lambda_{1})) \dots ((\lambda_{n}))$$
dans les situations suivantes.

(i) Le corps $k=\F$ 
est un corps fini 
de caract\'eristique diff\'erente de 3
contenant les racines cubiques de 1.

(ii) Le  corps $k$, de caract\'eristique diff\'erente de 3, poss\`ede une valuation discr\`ete,
par exemple $k$ est le  corps des fonctions d'une vari\'et\'e complexe de dimension au moins 1, ou est un corps $p$-adique, ou est un corps de nombres.  

  On trouve ainsi des hypersurfaces cubiques lisses non r\'etractes rationnelles
dans $\P_{\C(x_{1},\dots,x_{m})}^{n}$, avec un point rationnel, pour tout  entier $n$ avec $3 \leq n \leq m+2$.

En appliquant le th\'eor\`eme \ref{theodiagonalsurpadique}, 
sur un corps $k$ $p$-adique ($p\neq 3$) contenant une racine cubique de $1$,    on trouve
des hypersurfaces cubiques lisses non r\'etractes rationnelles
dans 
$\P_{k(x_{1},\dots,x_{m})  }^{n}$, avec un point rationnel, pour tout  entier $n$  avec
$4 \leq n \leq m+4$.

\section{Comparaison des r\'esultats obtenus par les diverses m\'ethodes}\label{comparaison}

Les  hypoth\`eses faites sur le corps $k$ dans chacun des trois derniers paragraphes diff\`erent.
Pour comparer la qualit\'e des r\'esultats qu'ils produisent, on consid\`ere
la situation sur le corps $E_{n} : = \C((t_{1}))\dots ((t_{n}))$. Comme on va voir,
aucune des trois m\'ethodes ne donne des r\'esultats enti\`erement  couverts par les deux autres.

En outre  les m\'ethodes des \S  \ref{reel}, \ref{ChatziLevine}
 et \ref{viadphi} donnent des hypersurfaces cubiques avec un indice de torsion  (comme d\'efini dans  \cite{CL}) \'egal \`a 2
 et celle du \S \ref{diagonal} donne  des exemples avec un indice de torsion \'egal \`a 3.

\medskip

La m\'ethode du    \S \ref{ChatziLevine} (\ChaL) fournit  des hypersurfaces cubiques  $X \subset \P^{N}_{E_{n}}$,
avec un point rationnel, non r\'etractes rationnelles, 
pour $N$ entier avec $N\geq 3$ de la forme $N=2^{\ell} \leq 2^{n-1}$.

La m\'ethode du \S \ref{viadphi}  fournit  des hypersurfaces cubiques  $X \subset \P^{N}_{E_{n}}$,
avec un point rationnel, non r\'etractes rationnelles, 
pour tout $N$ entier avec $3 \leq N \leq  2^{n-2}+1$.

La m\'ethode du \S \ref{diagonal} fournit des hypersurfaces cubiques  $X \subset \P^{N}_{E_{n}}$,
avec un point rationnel, non r\'etractes rationnelles, 
pour tout $N$ entier avec $3 \leq N \leq  n+2$.

\medskip

On obtient ainsi des exemples de telles hypersurfaces cubiques $X \subset \P^{N}_{E_{n}}$
pour les valeurs suivantes de $N$.

Pour $n=1$, les m\'ethodes classiques donnent des exemples non universellement $CH_{0}$-triviaux
avec $N=3$.
On a donc de tels exemples avec $N=3$ pour tout $n \geq 1$.

Pour $n=2$, on a d\'ej\`a $N=3$.
 La  m\'ethode du \S  \ref{diagonal} donne  des exemples non universellement $CH_{0}$-triviaux avec $3 \leq N  \leq 4$.
Un exemple  avec $N=4$ avait \'et\'e obtenu par Madore \cite{madore}, qui a montr\'e que pour l'hypersurface cubique $X$  d'\'equation
$$ x^3+y^3+\lambda z^3+ \mu u^3 + \lambda \mu v^3=0$$
sur le corps $\C((\lambda))((\mu))$ le groupe de Chow r\'eduit $A_{0}(X)$ n'est pas nul.

On a donc de tels exemples avec $N=4$ pour tout $n \geq 2$.

Pour $n=3$, on a d\'ej\`a $N=3, 4$.
 La m\'ethode du \S \ref{ChatziLevine} donne  des exemples non universellement $CH_{0}$-triviaux pour $N=4$, 
celle du \S  \ref{viadphi}  donne des exemples non r\'etractes rationnels avec $N=3$,
celle du \S \ref{diagonal}  donne  des exemples non universellement $CH_{0}$-triviaux pour
$N \leq 5$. Le cas $N=5$ est nouveau.
On a donc de tels exemples avec $N=5$ pour tout $n \geq 3$.

Pour $n=4$, on a d\'ej\`a $N=3,4,5$.
La m\'ethode du \S \ref{ChatziLevine} donne des exemples non universellement $CH_{0}$-triviaux pour $N=4$ 
 et $N=8$ (nouveau cas)
  celle du \S \ref{viadphi} donne des exemples non r\'etractes rationnels avec
$N \leq 5$, cas d\'ej\`a obtenu, 
celle du \S \ref{diagonal} donne des exemples non universellement $CH_{0}$-triviaux avec $N \leq 6$. 
Le cas $N=6$ est nouveau.
La situation reste ouverte pour $N=7$ et $N > 8$.
On a donc de tels exemples avec $N=6, 8$ pour tout $n \geq 4$.

A partir de $n=5$, on a $n+2 \leq 2^{n-2}+1$.
 On obtient des exemples non r\'etractes rationnels avec $N \leq  2^{n-2}+1$ par la m\'ethode 
  du \S \ref{viadphi}  et  des exemples non universellement $CH_{0}$-triviaux $N=2^{\ell} \leq 2^{n-1}$ par la m\'ethode du \S \ref{ChatziLevine}.

\medskip

Soit maintenant $k$ un corps $p$-adique et $F=k((t_{1})) \dots ((t_{n}))$. 
On obtient $X \subset \P^N_{F}$, $N \geq 3$,  hypersurface cubique lisse avec un point rationnel et non r\'etracte rationnelle pour
$N \geq 3$ satisfaisant l'une des conditions suivantes :

$N=2^{\ell} \leq 2^{n+1}$,  par la m\'ethode du \S \ref{ChatziLevine} (\ChaL);

$N \leq 2^n + 1$ , par la m\'ethode  du \S \ref{viadphi} ;

 $N \leq n+4$,
pour $p\neq 3$ et $k$ contenant les racines cubiques de 1  par la m\'ethode du \S 4 (Th\'eor\`eme \ref{theodiagonalsurpadique}).

Le cas des hypersurfaces cubiques lisses dans $\P^4_{k}$
sur $k$ un corps $p$-adique quelconque \cite[Th\'eor\`eme 1.19]{CTP} n'est  pas couvert par les
r\'esultats du pr\'esent article.

 \medskip
 
Question : Pour un corps $F$ donn\'e, l'ensemble des entiers $n\geq 3$ tels qu'il existe une
hypersurface cubique lisse dans $\P^n_{F}$  avec un $F$-point et non universellement $CH_{0}$-triviale est-il
un intervalle dans les entiers ?

\bigskip

{\it Remerciements.}  L'expos\'e de Marc Levine au Colloque International de K-Th\'eorie \`a Mumbai,
en janvier 2016,
et l'article de  Chatzistamatiou et Levine \cite{CL} m'ont amen\'e \`a ce travail. 
Le contenu du \S 4 a \'et\'e trouv\'e \`a l'occasion de la rencontre EDGE 2016 \`a Edimbourg (juin 2016), rencontre o\`u j'ai expos\'e les r\'esultats de l'article.
Je remercie Alena Pirutka  pour diverses remarques.
Je remercie  l'IRSES Moduli et l'Institut Tata (Mumbai)  pour leur soutien \`a l'occasion du colloque de Mumbai.
Ce travail a b\'en\'efici\'e d'une aide de l'Agence Nationale de la Recherche portant la r\'ef\'erence ANR-12-BL01-0005.

\end{document}